\documentclass[11pt]{article}

\usepackage{amsmath,amssymb,amsthm,amsfonts,verbatim}
\usepackage{enumerate}
\usepackage{microtype}
\usepackage[all,2cell,dvips]{xy}
\usepackage{mathtools}
\CompileMatrices

\usepackage[top=1.2in,bottom=1.2in,left=1.21in,right=1.21in]{geometry}

\title{Moving homology classes in finite covers of graphs}

\author{Benson Farb and Sebastian Hensel \thanks{The first author
    gratefully acknowledges support from the National Science
    Foundation.} }
    
\theoremstyle{plain}
\newtheorem{theorem}{Theorem}[section]

\newtheorem{lemma}[theorem]{Lemma}
\newtheorem{claim}[theorem]{Claim}

\newtheorem{observation}[theorem]{Observation}

\theoremstyle{definition}

\newcommand{\nc}{\newcommand}
\nc{\dmo}{\DeclareMathOperator}

\nc{\Q}{\mathbb{Q}}
\nc{\R}{\mathbb{R}}
\nc{\Z}{\mathbb{Z}}
\nc{\C}{\mathbb{C}}
\nc{\cS}{\mathcal{S}}
\nc{\iso}{\cong}
\dmo{\Mod}{Mod}
\dmo{\Diff}{Diff}
\dmo{\Homeo}{Homeo}
\dmo{\dist}{dist}
\dmo\BDiff{BDiff}
\dmo\SO{SO}
\dmo\slide{sl}
\dmo\im{im}
\dmo\Irr{Irr}
\dmo\Irrpr{Irr^{pr}}
\dmo\Fix{Fix}
\dmo\Out{Out}

\renewcommand{\epsilon}{\varepsilon}
\nc{\coloneq}{\mathrel{\mathop:}\mkern-1.2mu=}
\nc{\margin}[1]{\marginpar{\scriptsize #1}}
\nc{\para}[1]{\bigskip\noindent\textbf{#1}}

\begin{document}

\maketitle

\begin{abstract}
Let $Y\to X$ be a finite normal cover of a wedge of $n\geq 3$ circles.  We prove that for any $v\neq 0\in H_1(Y;\Q)$ there exists a lift $\widetilde{F}$ to $Y$ of a homotopy equivalence $F:X\to X$ so that the set of iterates $\{\widetilde{F}^d(v): d\in \Z\}\subseteq
H_1(Y;\Q)$ is infinite.  The main 
achievement of this paper is the use of representation theory to prove the existence
  of a purely topological object that seems to be inaccessible via
  topology.
\end{abstract}

\section{Introduction}
The group $\Out(F_n)$ of outer automorphisms of the free group $F_n$
is isomorphic to the group of homotopy classes of homotopy
equivalences of a wedge $X$ of $n$ circles.  For any finite regular
cover $f:Y\to X$ with covering group $G$, consider the finite-index
subgroup $\Gamma_G<\Out(F_n)$ of those homotopy equivalences which lift to $Y$
and act trivially on $G$.  There is a representation $\Gamma_G\to {\rm
  GL}(H_1(Y;\Q))$.  Analogous representations of mapping class
groups are known as {\em higher Prym representations} (see
\cite{Lo,PW}).

As one varies $G$, this gives a rich and useful family of
representations of $\Gamma_G$.  These have been studied by
Grunewald-Lubotzky and others; see \cite{GL} and the references
contained therein.  A basic {\em topological} problem that arises from
this situation is: given an arbitrary finite normal cover $Y\to X$
with covering group $G$, must it be true that the $\Gamma_G$-orbit of
any vector $v$ in $H_1(Y;\Q)$ is infinite?  The corresponding problem
with $X$ replaced by a surface and $\Out(F_n)$ replaced by the mapping
class group was posed by Putman-Wieland \cite{PW}.  They proved that
this problem is essentially equivalent to the vanishing of the virtual first Betti
number of the mapping class group, a major open problem posed by
Ivanov; see \cite{PW} for a detailed discussion.

For very special $G$-covers $Y\to X$, the representation of
$\Gamma_G<\Out(F_n)$ on $H_1(Y;\Q)$ can be worked out explicitly,
giving a positive answer to this question; see \cite{GL,GLLM}, and
\cite{Lo,Mc} for related work on finite abelian covers of surfaces.
For arbitrary $G$ it is not feasible to identify the actual
representation of $\Gamma_G$ explicitly.  Instead, one must construct
individual automorphisms of $F_n$ to move a given vector in
$H_1(Y;\Q)$.

One of the aspects that makes the Putman-Wieland problem challenging
is that the elements of $\Gamma_G$ act on $Y$ in a highly constrained
way: each must intertwine the $G$-action. However, for general covers
it is far from clear (and indeed probably false in general) that every automorphism of $H_1(Y;\Q)$ which
{\em homologically} intertwines correctly with $G$ actually arises as a
lift!  This makes it difficult to understand and control
$\Gamma_G$-orbits of vectors. 

In this paper we solve the
Putnam-Wieland problem for graphs.

\begin{theorem}[{\bf Moving homology classes in covers}]
\label{theorem:main-moving}
Let $f:Y\to X$ be a finite $G$-cover with ${\rm rank}(\pi_1(X))\geq
3$.  Then the $\Gamma_G$-orbit of any $0\neq v\in H_1(Y;\Q)$ is
infinite.  In fact, given any $0\neq v\in H_1(Y;\Q)$, there exists a
homotopy equivalence $F:X\to X$ such that $F$ lifts to
$\widetilde{F}:Y\to Y$ and so that the set $\{\widetilde{F}^d(v): d\in
\Z\}\subseteq H_1(Y;\Q)$ is infinite.
\end{theorem}

The statement of Theorem~\ref{theorem:main-moving} is false if ${\rm
  rank}(\pi_1(X))=2$:  we prove in \S~\ref{sec:ell-complicated-word} that the commutator of a 
  pair of generators of $\pi_1(X)$ is fixed under all of $\Out(F_2)$, and we use this to constuct a homology class with finite orbit.  In a similar vein, we want
to mention that Carlos Matheus \cite{M} has observed that there is 
an example of a cover of a genus $2$ surface
in which the Putman--Wieland problem is false. 

\paragraph{The key idea of the proof of Theorem~\ref{theorem:main-moving}.}  
To prove Theorem~\ref{theorem:main-moving} we must, given an arbitrary
$v\neq 0\in H_1(Y;\Q)$, construct out of nowhere a homotopy
equivalence $F:X\to X$ with the desired properties.  We first reduce
this problem to the problem of finding a loop $\ell$ in $X$ satisfying
some very special properties; we then show that an ``edge-slide''
along $\ell$ is the required $F$.  Even in simple examples the loop
$\ell$ that actually works is extremely complicated (see \S~\ref{sec:ell-complicated-top}), 
so how can we
find it in general? The key is to use the structure of
$H_1(Y;\Q)$ as a $G$-representation. This representation was computed in 1934 by  
Chevalley-Weil \cite{CW} for surfaces, and later by 
Gasch\"{u}tz for graphs (see \cite{GLLM}). 

\begin{theorem}[{\bf Chevalley-Weil for graphs} \cite{CW, GLLM}]
\label{thm:cw}
Let $X$ be a finite graph with $\pi_1(X)$ free of rank $n\geq 2$.  Let $Y\to X$ be a finite cover with deck group $G$.  Then as $G$-representations: 
\begin{equation}
\label{eq:cw:graphs}
H_1(Y;\Q) \iso \Q[G]^{n-1}\oplus \Q
\end{equation}
where $\Q[G]$ denotes the regular representation and $\Q$ the trivial representation of $G$.  
\end{theorem}

Using this information about the $G$--representation $H_1(Y;\Q)$ we
construct an $\ell$ so that the corresponding edge-slide map has the 
desired properties.  The novelty here is that we have
constructed a topological object non-explicitly, via representation theory, that
seems inaccessible via purely topological methods. For example, in most covers the loop $\ell$ cannot
be represented by a simple closed curve in \emph{any} identification of
the free group $F_n$ with the fundamental group of a punctured surface (see
\S~\ref{sec:ell-complicated-top}). 

\section{Proof of the main theorem}
\label{sec:moving}

In this section we prove Theorem~\ref{theorem:main-moving} in a sequence of steps. 
We begin by setting up some notational conventions that are used throughout the article.

Fix $n\geq 3$ and let $X$ be the wedge of $n$ copies of $S^1$. Let $p$ be the
point at which the circles are joined. We pick once and for all an
orientation on each copy of $S^1$, and we label these oriented edges
of $X$ by $a_1,\ldots ,a_n$. We identify $a_i$ with the corresponding
generator of $\pi_1(X,p) = F_n$.  Let $f:Y\to X$ be any finite normal cover of $X$ and let 
$G=\pi_1(X)/\pi_1(Y)$ be the corresponding deck group. We denote the
quotient homomorphism by $q:F_n\to G$. 

The vertices of the graph $Y$ are exactly the elements in the preimage
$f^{-1}(p)$ of the (unique) vertex of $X$. Each edges of $Y$ maps by
$f$ to a petal $a_i$ of $X$ injectively except at the endpoints.
Choose a preferred basepoint $\hat{p}\in f^{-1}(p)$, which we
denote as the \emph{identity vertex}. We label each other vertex
$\hat{p}' \in f^{-1}(p)$ by the element $g\in G$ such that
$g\hat{p} = \hat{p}'$.  Each vertex $\hat{p}'$ of $Y$ is of valence $2n$ since a small
neighborhood of the basepoint $p$ in $X$ lifts homeomorphically to
$Y$. We label each half-edge adjacent to $\hat{p}'$ with $a_i$
or $a_i^{-1}$, depending on which half-edge of $X$ it lifts
(interpreted as maps $[0,1]\to X$). Observing that the action of the deck group is
the action obtained by path-lifting yields the following
\begin{observation}\label{observation:cayley-graph}
  $Y$ is the Cayley graph of the group $G$
  with respect to the generating set $q(a_1),\ldots,q(a_n)$.
\end{observation}
Let us emphasize that the generating set $q(a_1),\ldots,q(a_n)$ is
not symmetric (and may contain identity elements) and thus the graph
$Y$ only has edges corresponding to right multiplication of the
$q(a_i)$ (and may have loops). 

\bigskip
We are now ready to begin the proof in earnest.

\para{Step 1 (A cohomology class that pairs with $v$): } Fix any edge $e\in Y$.  Let $\xi_e$ be the cocycle defined as follows. Given an arbitrary $1$-cycle $\sigma$, write it as a weighted
sum of edges $\sigma=ce+\sum_{e'\neq e}c_{e'}e'\in C_1(Y)$ and define $\xi_e(\sigma):=c$. 
The set $\{\xi_e: \text{$e$ is an edge of $Y$}\}$ spans $H^1(Y;\Q)$.   Since each edge of $Y$ is the lift of some $a_i$, it follows that we can relabel the $a_i$ so that there exists a lift $\alpha$ of $a_1$ to $Y$  so that $\xi_\alpha(v)\neq
0$. 

\para{Step 2 (Edge sliding maps and the action of their lifts): } Let $\ell\subset X$ be any loop that
intersects $a_1$ only in the basepoint $x_0\in X$.  We define a
homotopy equivalence $\slide_\ell:X\to X$ by taking the initial point
of $a_1$ and dragging it along $\ell$.  On the level of $\pi_1(X)$, we
have that $(\slide_\ell)_\ast$ maps $a_1$ to $\ell \ast a_1$ and fixes all
other generators $a_2,\ldots ,a_n$ of $\pi_1(X)$. Since $X$ is an
Eilenberg-MacLane space, the automorphism $\slide_\ast$ determines the homotopy equivalence $\slide_\ell$ up to homotopy.

\begin{lemma}[{\bf Action of a lifted edge-slide on \boldmath$H_1$}]
Suppose $\ell$ is a based loop in $X$.  If $\ell$ lifts to a closed
loop $\tilde{\ell}$ based at the identity vertex in $Y$ then
$\slide_\ell$ lifts to a homotopy equivalence $F:Y\to Y$ such that
$F_\ast:H_1(Y;\Q)\to H_1(Y;\Q)$ acts as   
\begin{equation}
\label{equation:action1}
 F_\ast(w)= w+\sum_{g\in G}\xi_{g\cdot \alpha}(w) [g\cdot \tilde{\ell}] 
 \end{equation}
 where $[g\cdot \tilde{\ell}]$ denotes the class of $g\cdot \tilde{\ell}$  in $H_1(Y;\Q)$.
\end{lemma} 

\begin{proof}
  The homotopy equivalence $\slide_\ell$ fixes the basepoint of
  $X$. Recall that we have a surjection $q:F_n \to G$ defining the
  cover $Y$. Since by assumption $\ell$ lifts to a closed loop in
  $Y$, we have $q(\ell) = 0$. Therefore $(\slide_\ell)_*:F_n\to F_n$
  preserves the kernel $\ker(q)$.  As a consequence, $\slide_\ell$ lifts to $Y$. 
  
  Let $F$ denote the lift of
  $\slide_\ell$ that fixes the identity vertex of $Y$. Since
  $q(\ell)=0$ it follows that $(\slide_\ell)_*$ induces the identity
  automorphism $G\to G$. Thus, for every 
  element $g\in G$ in the deck group, 
  \begin{equation}
  F(g\cdot x) = \slide_*(g)\cdot F(x) = g\cdot F(x) \label{equation:deck-action}
  \end{equation}
  In particular, $F$ fixes each vertex of $Y$ because $X$ has a single
  vertex and therefore each vertex of $Y$ is the image of the identity
  vertex under some $g\in G$.

Let $Y_0$ denote the connected component of $Y-{\rm interior}(p^{-1}(a_1))$
containing the identity vertex.  The restriction $f:Y_0\to X-a_1$
  is a covering map, and $\slide_\ell$ restricts to the identity map on 
  $X-a_1$. Thus $F|_{Y_0}$ is a lift of the identity that fixes
  the identity vertex $\widetilde{b}$ in $Y$, and so $F$ restricts
  to the identity on $Y_0$. By 
  Equation~(\ref{equation:deck-action}) $F$ restricts to the identity
  on each component $Y_i$ of $Y-p^{-1}(a_1)$ because each $Y_i$ is the
  $g$-translate of $Y_0$ for some $g\in G$. This implies that each
  edge of $Y$ labeled by any $a_i\neq a_1$ is also fixed by $F$,  
  since any such edge is contained in some $Y_i$.
  
  \smallskip
  Now consider the (unique) oriented edge $\widetilde{a}_1$ based at
  $\widetilde{b}$ which is a lift of $a_1$. The definition of lifting a map via lifting
  paths shows that $F$ maps $\widetilde{a}_1$
  to the edge-path $\widetilde{\ell}\ast \widetilde{a}_1$.  Equation~(\ref{equation:deck-action}) 
  implies that $F$ maps $g\cdot \widetilde{a}_1$
  to the edge-path $(g\cdot \widetilde{\ell})\ast (g\cdot
  \widetilde{a}_1)$.

  Since $F$ fixes each vertex in $Y$ and maps each edge to an
  edge-path, $F$ induces a chain map $F_\#$ on the chain complex
  $C_*(Y)$.  The discussion above implies that, for each oriented edge $\gamma$ of $Y$,
  \[ F_\#(\gamma) = 
  \left\{ 
  \begin{array}{ll}
  \gamma & \text{\ \ if $\gamma$ is
    labeled by $a_j, j>1$}\\
    \gamma + g\cdot \widetilde{\ell}&\text{\ \ if $\gamma = g\cdot \widetilde{a}_1$}
   \end{array}
   \right.\]
  Thus for each oriented edge $\gamma$ : 
  \begin{equation}
  F_\#(\gamma) = \gamma + \sum_{g\in G}\xi_{g\cdot \alpha}(\gamma)
  [g\cdot \tilde{\ell}]\label{equation:chain-level} 
  \end{equation}
  because $\xi_{g\cdot \alpha}(\gamma) = 1$  if $\gamma =
  g\cdot \alpha$ and $\xi_{g\cdot \alpha}(\gamma) =0$ otherwise.

  Because the oriented edges of $Y$ form a basis of $C_1(Y)$ and
  $F_\#$ is a chain map, Equation~(\ref{equation:chain-level}) holds
  for all $\gamma\in C_1(Y)$.  Taking the induced map $F_*$ on $H_1(Y;\Q)$ gives 
  Equation~(\ref{equation:action1}), as desired.
\end{proof}

\para{Step 3 (The right choice of $\ell$ implies the theorem): } Suppose that we can find a based loop $\ell\subset X$ satisfying each of the following:
\begin{enumerate}
\item $\ell$ does not intersect the interior of the edge $a_1$.
\item $\ell$ lifts to a closed loop $\tilde{\ell}$ based at the identity vertex in $Y$.
\item  The set
$\{G\cdot [\tilde{\ell}]\}\subset H_1(Y)$ is linearly
independent. 
\end{enumerate}

Under these conditions, $\slide_\ell$ lifts to a homotopy equivalence $F:Y\to Y$, and we will prove below that 
\begin{equation}
  \label{equation:action2}  
  F^n(w) = w+\sum_{g\in G}n\xi_{g\cdot \alpha}(w) [g\cdot \tilde{\ell}] 
\end{equation}
for all $n>0$ 

Since $\xi_{\alpha}(v)\neq 0$, Equation~\eqref{equation:action2} with $w=v$ implies that the coefficient of $[\ell]$ in $F^n(v)$ is unbounded as $n$ increases.  The assumption that 
$\{G\cdot [\tilde{\ell}]\}\subset H_1(Y)$ is linearly
independent gives that all of the $F^n(v)$ are distinct, thus proving the theorem.

\smallskip
To prove that Equation~\eqref{equation:action2} holds, first 
note that $\xi_{g\cdot \alpha}(g'\cdot\widetilde{\ell}) = 0$ for all
  $g,g'\in G$  since no lift of $\ell$ intersects any lift of the edge
  $a_1$. Thus $\xi_{g\cdot \alpha}(F(w)) = \xi_{g\cdot \alpha}(w)$ by
  Equation~(\ref{equation:action1}).

 Equation~(\ref{equation:action1}) and induction on $n$ now give: 
  \begin{eqnarray*}
   F^{n+1}(w) = F(F^n(w)) &=& F^n(w)+\sum_{g\in G}\xi_{g\cdot \alpha}(F^n(w))
   [g\cdot \tilde{\ell}]  \\
   &=& F^n(w)+\sum_{g\in G}\xi_{g\cdot \alpha}(w)[g\cdot \tilde{\ell}]
   \\
   &=& w+\sum_{g\in G}n\xi_{g\cdot \alpha}(w) [g\cdot \tilde{\ell}]  +\sum_{g\in G}\xi_{g\cdot \alpha}(w)[g\cdot \tilde{\ell}]  \\
   &=& w+\sum_{g\in G}(n+1)\xi_{g\cdot \alpha}(w) [g\cdot \tilde{\ell}]
  \end{eqnarray*}
and so Equation~\eqref{equation:action2} follows by induction.

\para{Step 4 (Finding the $\ell$): } By Step 3, to prove the theorem it is enough to find a based loop $\ell\subset X$ satisfying the three assumptions 1-3 stated in Step 3.  We do this via 
representation theory.

Recall that we have a surjection $q:F_n\to G$.  Let \[G_0:=q(\langle a_2,\ldots ,a_n\rangle)\]

Let $Y_0$ denote the connected component of $Y-{\rm interior}(p^{-1}(a_1))$ containing the identity vertex.  The restriction of the covering map $f:Y\to X$ to $Y_0$ is a regular cover with deck group $G_0$.  Thus, by Theorem~\ref{thm:cw}, we have the following.

\begin{observation}
$H_1(Y_0;\Q)$ is isomorphic as a $G_0$-module to $\Q[G_0]^{n-2}\oplus \Q$.
\end{observation}

Now note that $Y-{\rm interior}(p^{-1}(a_1))$ is the union of the translates $Y_0,Y_1,\ldots , Y_k$  of the subgraph $Y_0$ under the deck group $G$; here $k=[G:G_0]$.  The reason for this is that $p^{-1}(a_1)$ is $G$-invariant, and hence so is its complement.

\begin{claim}\label{claim:yj-independent}
The map 
\[I:\oplus_{j=0}^kH_1(Y_j) \to H_1(Y)\]
induced by the inclusions $Y_j\to Y$ is injective.
\end{claim}
\begin{proof}
  First note that the graphs $Y_j$ are disjoint, and each vertex of
  $Y$ is contained in some $Y_j$.

  By collapsing edges of $Y$ with distinct endpoints, we can ensure that each
  $Y_j$ is a rose (i.e. has a single vertex). Since such collapses are homotopy
  equivalences, this modification does not change injectivity of the
  map $I$. Similarly, we can collapse further edges in the complement
  of the $Y_j$ to ensure that $Y$ becomes itself a rose, again without
  changing $I$. In the resulting rose, the claim is clear since each
  $Y_i$ is a ``sub-rose'',  and any two of these sub-roses intersect only at the one remaining vertex.
\end{proof}

Now choose any $u\in H_1(Y_0;\Q)$ so that $\{G_0\cdot u\}$ is linearly independent;
this is possible by considering $u \in \Z[G_0] \subset \Q[G_0]^{n-2}\oplus \Q \iso H_1(Y_0;\Q)$
of the form $u=1\cdot g_0$ for any $g_0\in G_0$.

\begin{claim}
  The set $\{g\cdot u : g\in G\}\subset H_1(Y;\Q)$ is linearly independent.
\end{claim}

\begin{proof}
  Suppose that
  \[ \sum_{g \in G} c_g gu = 0 \] 
  We rewrite this as a sum over coset representatives,
  \[ \sum_{h \in G/G_0}\sum_{g_0\in G_0} c_{hg_0} hg_0u = 0 \] 
  By Claim~\ref{claim:yj-independent} we then see that for each $h \in G/G_0$
  we have
  \[ 0 = \sum_{g_0 \in G} c_{hg_0} hg_0u = h\left(\sum_{g_0 \in G} c_{hg_0} g_0u\right) \] 
  and thus
  \[ 0 = \sum_{g_0 \in G} c_{hg_0} g_0u \] 
  By the defining property of $u$ this however implies $c_{hg_0} = 0$ for all $g_0\in G_0$.
  Since this is true for all $h,g_0$ we see that  $c_g = 0$ for all $g\in G$.
\end{proof}

Hence, any loop $\ell \subset Y_0$ defining the element $u \in H_1(Y_0)$ 
has the desired properties 1-3. This finishes the proof of Theorem~\ref{theorem:main-moving}. 

\section{Complexity of $\ell$}
\label{sec:ell-complicated}

Even in the simplest cases, the simplest loop $\ell$ with the
required properties 1-3 is necessarily complicated from a topological
perspective (and the complexity grows with the
degree of the cover). By using representation theory, we were able
to completely bypass this issue, but in this section we collect
two observations which indicate the complexity of the desired element.

\subsection{$\ell$ is complicated as a word}
\label{sec:ell-complicated-word}
In this section we describe explicitly a simple but nontrivial example
of the homology of a cover as a representation. The purpose of this
example is twofold -- on the one hand it will show that even in the
simplest cases, words $\ell$ as required in the proof of Theorem~\ref{theorem:main-moving} 
 are fairly complicated.  As a byproduct, we will show that
Theorem~\ref{theorem:main-moving} is false in rank $2$.

\smallskip
Let $X$ be the wedge of two circles. To simplify notation, denote the
generators of $\pi_1(X)$ by $a,b$.  Let $Y$ be the cover of $X$ corresponding to the kernel of the map $q:\pi_1(X) \to H_1(X;\Z/2\Z)$. This is a $4$-fold cover with 
deck group $\Z/2\Z \times \Z/2\Z$.
\begin{figure}
  \centering
  \includegraphics[width=0.2\textwidth]{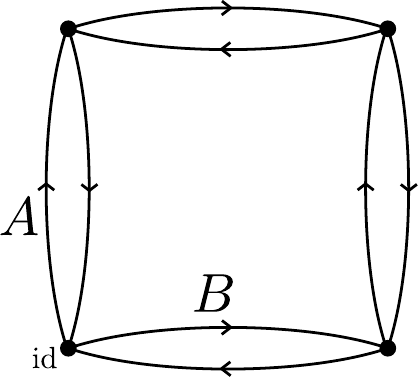}
  \caption{The mod-2 homology cover}
  \label{fig:cover-mod2}
\end{figure}
Theorem~\ref{thm:cw} implies that 
\[H_1(Y;\Q) = \Q^2 \oplus V_a \oplus V_b \oplus V_{ab}\] 
where
$V_a,V_b,V_{ab}$ are the three nontrivial rational 
irreducible representations of $G=\Z/2\Z\times\Z/2\Z$. Each of these representations is $1$-dimensional, and the representations are given as follows:
\[ q(a)|_{V_a}(z) = z, \quad q(b)|_{V_a}(z) = -z \] 
\[ q(a)|_{V_b}(z) = -z, \quad q(b)|_{V_b}(z) = z \] 
\[ q(a)|_{V_{ab}}(z) = -z, \quad q(b)|_{V_{ab}}(z) = -z \] In order to
find these representations topologically, we begin by considering the
lifts $A, B$ of $a^2, b^2$ respectively to the identity
basepoint. We then see the following
\[ V_a = \langle A - q(b)A \rangle \]
\[ V_b = \langle B - q(a)B \rangle \]
which follows simply as the representation is the correct one. To
describe $V_{ab}$, let $C$ be the elevation of $ab$ at the identity
vertex and let $C'$ be the
elevation of $ba$ at the identity vertex.  It is easy to check that 
\[ V_{ab} = \langle C - C' \rangle .\]
Note that each of these three irreducible representations can be generated by linear combinations
of elevations of primitive elements -- but not by elevations of primitive
elements themselves. This is indicative of the situation for general covers. 

The trivial representation is the image of the transfer map:
$$\Q^2 = \langle A + q(b) A, B + q(a) B \rangle$$ 

As a final note, we will describe an element $x \in H_1(Y;\Q)$ so that
$\langle Gx \rangle \simeq \Q[G]$ (such an element is not unique). Such an $x$ 
necessarily needs to contain a contribution from each irreducible representation; here
we can simply add the bases of
$V_a, V_b, V_{ab}$ and something nonzero in the transfer part. 
One example is the lift of the element
\[ a^3b^{-1}ab. \]
Note that, again, this element is not primitive, but $x$ is a linear
combination of elevations of primitive elements. In fact, elements with the
defining property of $x$ usually cannot be equal to elevations of primitive
elements (as will be shown in the next section).

Also note that for more complicated $G$, the complexity of $\ell$ will increase
further, as more irreducible representations appear in $\Q[G]$.

\smallskip
Finally, note that the loop representing the commutator $[a,b]$ lifts in $Y$ to 
a homologically nontrivial loop. The group $\Out(F_2)$ fixes the conjugacy class
defined by  $[a,b]$, and hence any lift of an element in  $\Out(F_2)$ permutes the
set of lifts of $[a,b]$. This shows
\begin{lemma}
  There are covers $Y\to X$ of a rank $2$ graph $X$, and homology classes
  $v\in H_1(Y;\Q)$ so that the set
  \[ \{\widetilde{\phi}_*(v) \,|\, \phi\in\Out(F_n)\mbox{ lifts to }Y\} \]
  is finite.
\end{lemma}

\subsection{$\ell$ is complicated as a topological object}
\label{sec:ell-complicated-top}

Here, we note the following simple observation
\begin{lemma}
  Let $Y \to X$ be a regular cover of graphs with the property that no primitive
  element in $\pi_1(X)$ lifts (with degree $1$) to $Y$. Then any element $\ell$
  as in the proof of Theorem~\ref{theorem:main-moving} cannot be represented by a simple closed curve
  with respect to any identification $\pi_1(X) \iso\pi_1(\Sigma)$ for a punctured 
  surface $\Sigma$. More generally, any element $\ell$ satisfying only property 3 cannot map in
  $\pi_1(X)$ to a multiple of a simple closed curve.
\end{lemma}
\begin{proof}
  Let the identification $\pi_1(X) \iso\pi_1(\Sigma)$ be given, and let $\Sigma'\to\Sigma$
  be the cover of $\Sigma$ defined by $\pi_1(Y)$. Note that any element
  given by a simple closed curve in $\pi_1(\Sigma)$ is primitive. Hence, given any simple closed
  curve $\gamma$, the preimage of $\gamma$ in $\Sigma'$ consists of strictly less than
  $|G|$ elements; hence the span of these elements is of rank strictly less than $|G|$. 
  Thus, condition 3. of the defining properties of $\ell$ cannot hold.
\end{proof}
We remark that covers as in the lemma are plentiful; in particular any cover that covers the
mod-$n$ homology cover of $X$ has the desired property.

\small

\begin{tabular}{ll}
Benson Farb & Sebastian Hensel\\
Department of Mathematics & Mathematisches Institut\\ 
University of Chicago & Rheinische Friedrich-Wilhelms-Universit\"at Bonn\\
5734 University Ave. & Endenicher Allee 60\\
Chicago, IL 60637 & 53115 Bonn, Germany\\
E-mail: farb@math.uchicago.edu & E-mail: hensel@math.uni-bonn.de
\end{tabular}



\begin{thebibliography}{ABCDEF}
\small

\bibitem[CW]{CW}
C. Chevalley and A. Weil, \"{U}ber das Verhalten der Integrale 1. Gattung bei Auto- morphismen des Funktionenk\"{o}rpers. {\em Abh. Math. Sem. Univ. Hamburg}, 10:358Ð361, 1934.

\bibitem[FM]{FM}
B. Farb and D. Margalit, A primer on mapping class groups, to appear in \emph{Princeton Mathematical Series}, Princeton University Press. Available at: http://www.math.utah.edu/~margalit/primer/

\bibitem[GLLM]{GLLM}
F. Grunewald, M. Larsen, A. Lubotzky and J. Malestein, Arithmetic quotients of the mapping class group, preprint, April, 2015.

\bibitem[GL]{GL}
F. Grunewald and A. Lubotzky, Linear representations of the automorphism group of a free group, {\em Geom. Funct. Anal. }18 (2009), no. 5, 1564Ð1608. 

\bibitem[L]{Lo}
E. Looijenga, Prym Representations of Mapping Class Groups, {\em Geometriae Dedicata} 64 (1997), 69--83.

\bibitem[M]{M}
C. Matheus, Some comments on the conjectures of Ivanov and Putman-Wieland, https://matheuscmss.wordpress.com/2015/04/24/some-comments-on-the-conjectures-of-ivanov-and-putman-wieland/ 


\bibitem[McM]{Mc}
C. T. McMullen, Braid groups and Hodge theory, {\em Mathematische Annalen} 355, no. 3 (2013), 893--946.

%
\bibitem[PW]{PW}
A. Putman and B. Wieland, Abelian quotients of subgroups of the
mappings class group and higher Prym representations, {\em
  J. Lond. Math. Soc.} (2) 88 (2013), no. 1, 79--96.


\end{thebibliography}
\end{document}